\documentclass{article}

\usepackage{amsmath,amssymb,amsthm}	

\usepackage{graphicx}		  

\newtheorem{thm}{Theorem}[section]
\newtheorem{lem}[thm]{Lemma}
\newtheorem{prop}[thm]{Proposition}
\newtheorem{cor}[thm]{Corollary}

\newtheorem{rem}{Remark}

\newtheorem{sublem}[thm]{Sublemma}
\newtheorem{main}[thm]{Main theorem}

\title{The number of cusps of right-angled polyhedra \\ in hyperbolic spaces}

\author{Jun Nonaka}

\date{}

\begin{document}

\maketitle
\begin{abstract}
As was pointed out by Nikulin \cite{ref:nik} and Vinberg \cite{ref:vin}, 
a right-angled polyhedron of finite volume in the hyperbolic $n$-space 
$\mathbb{H}^n$ has at least one cusp for $n\geq 5$. 
We obtain non-trivial lower bounds on the number of cusps of such polyhedra. 
For example, right-angled polyhedra of finite volume must have at least three cusps for $n=6$. 
Our theorem also says that the higher the dimension of a right-angled polyhedron becomes, the more cusps it must have. 
\end{abstract}

\textbf{Key words:} Cusp, Right-angled polyhedron, Hyperbolic space, Combinatorics. \\

\textbf{MSC (2010)} 20F55, 51F15, 57M50\\

\section{Introduction}
\indent Let $P$ be a convex polyhedron in the hyperbolic $n$-space $\mathbb{H}^n$ 
with dihedral angles 
of the form $\pi /m$ $(m \in \mathbb{N} )$ at all its $(n-2)$-dimensional faces. 
We call this polyhedron a $Coxeter$ $polyhedron$. 
In particular, a polyhedron is called $right$-$angled$ if all its dihedral angles are $\pi /2$. 

\indent We call a polyhedron $acute$-$angled$ if all its dihedral angles do not exceed $\pi /2$. 
It is known that any $k$-dimensional face of an acute-angled polyhedron $P\subset \mathbb{H}^n$ 
belongs only to $(n-k)$ hyperfaces; any $k$-dimensional face is represented by the intersection of 
$(n-k)$ hyperfaces $($see \cite{ref:avs}$)$. 
In particular, any ordinary vertex belongs only to $n$ hyperfaces. 
If $P$ is a right-angled polyhedron, 
then the number of hyperfaces of $P$ which share one cusp is 
exactly $2(n-1)$. Any of these hyperfaces is parallel to one other 
and adjacent to the remaining $2(n-2)$ hyperfaces. 

\indent An $n$-dimensional combinatorial polytope is called $simple$ if any of its vertices belongs only 
to $n$ hyperfaces, and $simple$ $at$ $edges$ if any of its edges belongs only to $(n-1)$ hyperfaces. 
In addition, we call an $n$-dimensional hyperbolic polytope $almost$ $simple$ if it is simple at edges and 
any of its vertices not at infinity belongs only to $n$ hyperfaces. 
According to the above, any compact acute-angled polyhedron in $\mathbb{H}^n$ is simple, 
and any acute-angled polyhedron of finite volume with vertices at infinity added is simple at edges. 
In particular, any right-angled polyhedron of finite volume in $\mathbb{H}^n$ is almost simple. 

\indent Vinberg proved that there are no compact right-angled polyhedra in $\mathbb{H}^n$ for $n>4$ 
$($see \cite{ref:vin}$)$. 
On the other hand, Dufour proved the following theorem. 

\begin{thm}\label{thm:1-1} $(${\rm \cite{ref:du}}$)$
Right-angled polyhedra of finite volume may exist in $\mathbb{H}^n$ only if $n<13$. 
\end{thm}

\begin{rem}
{\rm Before Dufour proved Theorem \ref{thm:1-1}, Potyagailo and Vinberg had already shown 
the nonexistence of right-angled polyhedra of finite volume in $\mathbb{H}^n$ for $n>14$ in \cite{ref:pv}. }
\end{rem}

\indent These results suggest that when the dimension of a right-angled polyhedron 
in the hyperbolic space becomes higher, then this polyhedron becomes far from compact. 
Therefore we may expect that the higher the dimension of a right-angled polyhedron is, 
the more cusps it has for $4<n<13$. 
Let $Q^n$ be a right-angled polyhedron of finite volume in $\mathbb{H}^n$. 
Denote the number of cusps of $Q^n$ by $c(Q^n)$. 
Our main theorem shows that this expectation is actually true. 

\begin{main}\label{main}
\textit{For $Q^n$, a right-angled polyhedron of finite volume in $\mathbb{H}^n$, we have the following lower bounds on the number of cusps $c(Q^n)$: 
\begin{eqnarray*}
& c(Q^6)\geq 3, \, c(Q^7)\geq 17, \, c(Q^8)\geq 36, \, c(Q^9)\geq 91, \\ 
& c(Q^{10})\geq 254, \, c(Q^{11})\geq 741, \, c(Q^{12})\geq 2200. 
\end{eqnarray*}}
\end{main}

\indent Potyagailo and Vinberg \cite{ref:pv} found some examples of right-angled polyhedra 
of finite volume in $\mathbb{H}^n$ for $n\leq 8$. 
However, we do not know whether there exist right-angled polyhedra of finite volume 
in $\mathbb{H}^n$ for $9\leq n\leq 12$. 
According to \cite{ref:ft}, there is no simple ideal Coxeter polyhedron in $\mathbb{H}^n$ 
with $n\geq 8$. Furthermore in \cite{ref:k}, it was shown that there is no right-angled 
polyhedron in $\mathbb{H}^n$ with $n\geq 7$. \\
\indent The key element of the proof of the main theorem 
is a lower bound on the number of 2-dimensional faces of $Q^3$, which we prove in Section $3$. 
Before doing so, we introduce some known results on right-angled polyhedra in $\mathbb{H}^n$. 

\section{Right-angled polyhedra in $\mathbb{H}^n$} 

\indent Let us denote by $\mathbb{H}^n$ the hyperbolic $n$-space. 
There are some ways to describe it. 
In what follows, we use two of them that we now explain below. 
Denote the unit open ball by
\begin{equation*}
B^n:=\{ x\in \mathbb{R}^n\mid \mid x\mid <1\}. 
\end{equation*} 
The metric space consisting of $B^n$ equipped with a Riemannian metric of the form 
\begin{equation*}
\left( \frac{2}{1-\mid x\mid ^2}\right) ^2 \sum_{i=1}^n dx_i^2
\end{equation*}
is called $the$ $conformal$ $ball$ $model$ of $\mathbb{H}^n$. 
Denote the upper half-space by
\begin{equation*}
U^n:=\{ (x_1, \dots , x_n)\in \mathbb{R}^n\mid x_n>0\}. 
\end{equation*}
The metric space consisting of $U^n$ together with the metric 
\begin{equation*}
\frac{1}{x_n^2}\sum_{i=1}^n dx_i^2
\end{equation*}
is called $the$ $upper$ $half$-$model$ of $\mathbb{H}^n$. 

\indent Let $P^n$ be a Coxeter polyhedron of finite volume in the hyperbolic $n$-space $\mathbb{H}^n$. 
The boundary of $\mathbb{H}^n$ is the $(n-1)$-sphere $\mathbb{S}^{n-1}$ 
when we consider the conformal ball model. 
Because the volume of $P^n$ is finite, 
the intersection of the Euclidean closure of $P^n$ and $\mathbb{S}^{n-1}$ is a finite set of points. 
Here, the Euclidean closure of a set $A\subset \mathbb{H}^n$ means the closure of $A$ in $\mathbb{R}^n$ 
contained in $B^n$. 
A cusp point (or a point at infinity, an ideal vertex) of $P^n$ is a point $c$ of $\overline{P^n}\cap \mathbb{S}^{n-1}$. 
We say that $P^n$ $has$ $a$ $cusp$ when there is a cusp point of $P^n$.  
Denote the set of $k$-dimensional faces of $P^n$ by $\Omega _k(P^n)$. 
The number of $k$-dimensional faces of $P^n$ is denoted by $a_k(P^n)$, 
and the number of cusps of $P^n$ is denoted by $c(P^n)$. 
The expression \lq \lq the faces $F_1, \cdots , F_k$ intersect\rq \rq \ means that 
the faces $F_1, \cdots , F_k$ have a common point. 
We say these faces $share$ $a$ $cusp$ or $have$ $a$ $common$ $cusp$ 
when their Euclidean closures have a common point at infinity. 

\indent Two hyperplanes of $\mathbb{H}^n$ are called $parallel $ 
if they do not intersect. 
We call that two hyperfaces $F_1$ and $F_2$ of $P^n$ are $parallel $ 
if the hyperplanes containing them are parallel 
and their Euclidean closures intersect in the boundary of $\mathbb{H}^n$. 
If two hyperfaces are parallel, 
then the intersection of their Euclidean closures is exactly one cusp of $P^n$. 

\indent Let $Q^n$ be a right-angled polyhedron of finite volume in $\mathbb{H}^n$. 
A right-angled polyhedron has some good properties stated as $($P1$)$, $($P2$)$ and $($P3$)$ below. 

\indent $($P1$)$ Any face of $Q^n$ is also a right-angled polyhedron. 

\indent $($P2$)$ For any hyperface of $Q^n$ passing through a cusp $c$, 
there is a unique other parallel hyperface sharing same cusp. 

\indent $($P3$)$ The number of hyperfaces of $Q^n$ sharing a cusp of $Q^n$ is exactly $2(n-1)$. 

These properties follow from the local combinatorial structure of right-angled polyhedra in $\mathbb{H}^n$. 
In what follows, $Q^n$ always denotes a right-angled polyhedron of finite volume in $\mathbb{H}^n$. 

\indent We denote by $\langle F \rangle $ the Euclidean closure of 
the hyperplane containing a hyperface $F$ of $Q^n$. 

\begin{prop}\label{pro:2-1} $(${\rm \cite[p.7]{ref:pv}}$)$ 
Let $F_1, F_2, \cdots $ be hyperfaces of a right-angled polyhedron. Then

\indent $(a)$ if $\langle F_1 \rangle $ and $\langle F_2 \rangle $ intersect, 
then $F_1$ and $F_2$ intersect; 

\indent $(b)$ if $F_1$, $F_2$, $F_3$ are pairwise mutually adjacent, 
then they meet at a $(n-3)$-dimensional face; 

\indent $(c)$ if $F_1$ and $F_2$ are parallel and $F_3$ is adjacent to them, 
then $F_1$, $F_2$ and $F_3$ meet at a cusp; 

\indent $(d)$ if $F_1$ and $F_2$ are parallel and $F_3$ and $F_4$ are adjacent to them, 
then $F_1$, $F_2$, $F_3$ and $F_4$ meet at a cusp. 
\end{prop}

\indent Potyagailo and Vinberg proved certain inequalities describing the relations 
between the number of cusps and faces of a right-angled polyhedron in $\mathbb{H}^n$. 

\begin{prop}\label{pro:2-2} $(${\rm \cite[p.7, 8]{ref:pv}}$)$ 
Let $Q^n$ be a right-angled polyhedron of finite volume in $\mathbb{H}^n$. 
Then the following inequalities hold: 
\begin{eqnarray*}
a_1(Q^2)+c(Q^2) \geq 5,\ \indent a_2(Q^3) \geq 6, \ \indent a_2(Q^3)+2c(Q^3) \geq 12. 
\end{eqnarray*}
\end{prop}

The first inequality follows from the following obvious lemma. 

\begin{lem}\label{lem:2-3} 
If $Q^2$ is compact, then $Q^2$ has more than four edges. 
\end{lem}

\indent On the other hand, Nikulin considered 
the average number of $k$-dimensional faces in $l$-dimensional faces of an acute-angled polyhedron $P$ 
in $\mathbb{H}^n$, denoted by $a_k^l(P)$: 
\begin{equation*}
a_k^l(P)=\frac{1}{a_k(P)}\sum_{F\in \Omega ^k(P)}a_l(F). 
\end{equation*}
One of the main ingredients for proving our main theorem is the following Nikulin's inequality:

\begin{thm}\label{thm:2-4} 
Let $P$ be an acute-angled polyhedron of finite volume in $\mathbb{H}^n$. Then 
\begin{equation*}
a_k^l(P)
<\binom{n-l}{n-k}\frac{\binom{[\frac{n}{2}]}{l}+\binom{[\frac{n+1}{2}]}{l}}{\binom{[\frac{n}{2}]}{k}+\binom{[\frac{n+1}{2}]}{k}}
\end{equation*}
holds for $l<k\leq [\frac{n}{2}] $. 
\end{thm}

This inequality is proved by Nikulin \cite{ref:nik} for simple convex polyhedra 
and a generalization is due to Khovanskij \cite{ref:kh} for polyhedra simple at edges. 

By Nikulin's inequality and Lemma \ref{lem:2-3}, we obtain the following proposition, 
as explained in \cite{ref:vin}. 

\begin{prop}\label{pro:2-3}  
There are no compact right-angled polyhedra in $\mathbb{H}^n$ for $n>4$. 
\end{prop}

This proposition shows that $Q^n$ must have at least one cusp for $n\geq 5$. 

Fix a cusp $c$. Denote the set of all $k$-dimensional faces which contain $c$ by $\Omega ^c_k(Q^n)$. 
As we have mentioned in the introduction, 
the number of elements of this set is $2(n-1)$. 
In the next section, we focus on the case $n=3$. 

\section{A lower bound on the number of $2$-dimensional faces of $Q^3$}

In this section, we consider a right-angled polyhedron $Q^3$ in the upper half-model. 
In this model, the hyperplane must be a vertical Euclidean plane 
or an upper hemisphere which intersects the boundary of the upper half-space 
orthogonally. 
We assume that $Q^3$ has a cusp $c$. 
And we also assume that $c$ is the point at infinity of the upper half-space. 
The set $\Omega _2^c(Q^3)$ has exactly four elements: $A_0$, $A_1$, $A_2$, $A_3$. 
In this case, the $2$-dimensional faces $A_0$, $A_1$, $A_2$ and $A_3$ are parts of 
vertical Euclidean planes. 
We assume that $A_0$ and $A_2$ are parallel, and that $A_1$ and $A_3$ are parallel. 

\indent The aim of this section is to prove the following lemma. 

\begin{lem}\label{lem:3-1} 
If $Q^3$ has only one cusp, then $Q^3$ has at least twelve
$2$-dimensional faces. 
Moreover, if $Q^3$ has exactly twelve $2$-dimensional faces, 
then, after replacing $A_0$ with $A_1$ and $A_2$ with $A_3$, if necessary,  
$A_0$ and $A_2$ are quadrangles, and $A_1$ and $A_3$ are pentagons. 
\end{lem}

\begin{figure}[htp]
\begin{center}
\includegraphics[width=10cm]{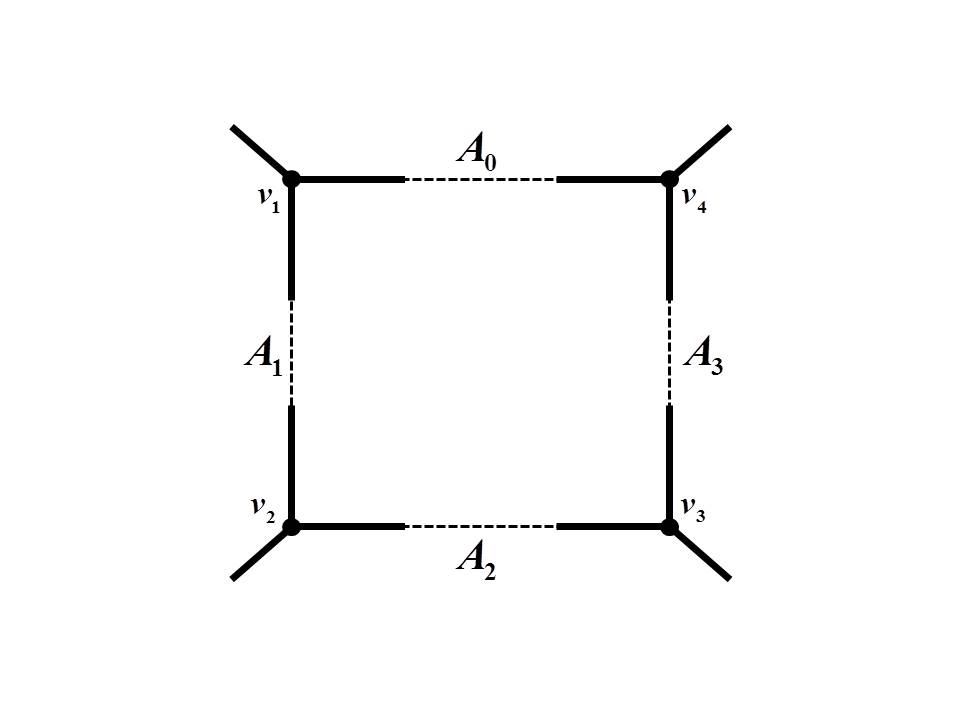}
\end{center}
\caption{Combinatorial structure of $Q^3$}
\label{fig:1}
\end{figure}

Assume that $c(Q^3)=1$, and denote this cusp by $c$. 
There are exactly four $2$-dimensional faces which share $c$. 
Denote these $2$-dimensional faces by $A_0$, $A_1$, $A_2$, $A_3$ as before. 
Because the faces $A_0$ and $A_1$ are adjacent, they have a common edge.
This edge starts from a vertex and terminates at $c$. We denote this
vertex by $v_1$. 
In the same manner, we denote by $v_2$ $($resp. $v_3$, $v_4)$ 
the vertex belonging to $A_1$ and $A_2$ $($resp. $A_2$ and $A_3$, $A_3$ and $A_0)$ 
which is an endpoint of an edge terminating at $c$. 
Then the local combinatorial structure of $Q^3$ around $c$ can 
be depicted as in Fig. \ref{fig:1}.

Each vertex $v_1$, $v_2$, $v_3$ and $v_4$ belongs to 
two non-compact $2$-dimensional faces and one compact 
$2$-dimensional face of $Q^3$ because $Q^3$ is almost simple and has only one cusp. 

Since $Q^3$ is a right-angled polyhedron in $\mathbb{H}^3$, it must satisfy the following five conditions: 

$($1$)$ Any compact $2$-dimensional face must have more than four edges. 

$($2$)$ If two $2$-dimensional faces are adjacent, then they have only one common edge. 

$($3$)$ Any vertex must be shared by exactly three edges. 

$($4$)$ There are no three $2$-dimensional faces which are pairwise adjacent but do not share a vertex. 

$($5$)$ There are no four $2$-dimensional faces which are cyclically adjacent 
except for $A_0$, $A_1$, $A_2$ and $A_3$. 

Condition $($1$)$ comes from Lemma \ref{lem:2-3}, and Conditions $($2-3$)$ 
can be seen from the local combinatorial structure of a right-angled 
polyhedron in $\mathbb{H}^3$. 
We use the following theorem due to Andreev to obtain Conditions $($4-5$)$. 

\begin{thm}\label{thm:And} $(${\rm \cite{ref:and}}$)$
An acute-angled almost simple polyhedron of finite volume with given dihedral angles, 
other than a tetrahedron or a triangular prism, exists in $\mathbb{H}^3$ 
if and only if the following conditions are satisfied$:$ 

$(a)$ if three $2$-dimensional faces meet at a vertex or a cusp, then the sum of the dihedral angles 
between them is at least $\pi $ $( \pi $ for a cusp$);$ 

$(b)$ if four $2$-dimensional faces meet at a vertex or a cusp, then all the 
dihedral angles between them equal $\frac{\pi}{2} ;$ 

$(c)$ if three $2$-dimensional faces are pairwise adjacent but share neither a vertex nor a cusp, 
then the sum of the dihedral angles between them is less than $\pi ;$ 

$(d)$ if a $2$-dimensional face $F_i$ is adjacent to $2$-dimensional faces $F_j$ and $F_k$, 
while $F_j$ and $F_k$ are not adjacent but have a common cusp which $F_i$ 
does not share, then at least one of the angles formed by 
$F_i$ with $F_j$ and with $F_k$ is different from $\frac{\pi}{2} ;$ 

$(e)$ if four $2$-dimensional faces are cyclically adjacent but meet at neither a vertex nor a cusp, 
then at least one of the dihedral angles between them is different from $\frac{\pi}{2}$. 
\end{thm}
Since all the dihedral angles of $Q^3$ are $\frac{\pi }{2}$, by Theorem \ref{thm:And} $(c)$, 
we obtain Condition $($4$)$. Moreover, by Theorem \ref{thm:And} $(b)$ and $(e)$, 
we obtain Condition $($5$)$. 

Before proving Lemma \ref{lem:3-1}, we have to prove some sublemmas. 
Now we prepare suitable notation which will be used in their proofs. 
Let $B_i$ be a compact $2$-dimensional face which is adjacent to $A_i$ and $A_{i+1}$ 
$($integer $i$ can be 0, 1, 2 or 3, and when $i=3$, we assume that $A_{i+1}=A_0$ $)$. 
One of the endpoints of the edge $A_i\cap A_{i+1}$ is a cusp, and the other is a vertex. 
Thus any $2$-dimensional face which is adjacent to both $A_i$ and $A_{i+1}$ must have this vertex. 
But, since $Q^3$ is almost simple, this vertex is shared only by three $2$-dimensional faces. 
That is to say, $B_i$ is the only $2$-dimensional face which is adjacent to both $A_i$ and $A_{i+1}$ 
for $i=0, 1, 2, 3$. Thus we obtain the following sublemma. 

\begin{sublem}\label{sublem:3-3} 
\textit{No compact $2$-dimensional face of $Q^3$ is adjacent to both $A_i$ and $A_{i+1}$ other than $B_i$. }
\end{sublem}%

On the other hand, by the following sublemma, we know that there is no compact $2$-dimensional face of $Q^3$ which is adjacent to both $A_i$ and $A_{i+2}$ for $i=0, 1$. 

\begin{sublem}\label{sublem:3-4} 
\textit{Assume that a right-angled polyhedra $Q^3$ has a cusp $c$. Let $A$ and $A'$ be $2$-dimensional faces of $Q^3$ which are parallel but sharing this cusp $c$. 
Then any $2$-dimensional face which is adjacent to both $A$ and $A'$ must also have the cusp $c$. }
\end{sublem}%

\begin{proof}
Since $Q^3$ is a right-angled polyhedron in $\mathbb{H}^3$, 
the number of $2$-dimensional faces of $Q^3$ which have $c$ is 4. 
We denote these faces by $A$, $A'$, $A''$ and $A'''$. 
We may assume that $A''$ and $A'''$ $($resp. $A$ and $A')$ are parallel. 
We may also assume that the cusp $c$ of $Q^3$ is identified with the point at infinity of the upper half-space. 
In this case, each of the hyperplanes $\langle A\rangle $, $\langle A'\rangle $, 
$\langle A''\rangle$ and $\langle A'''\rangle$ has to be a vertical Euclidean plane which 
intersects the boundary of $\mathbb{H}^3$ orthogonally. 

We suppose that there is a $2$-dimensional face which is adjacent to both $A$ and $A'$ 
other than $A''$ and $A'''$. We denote this face by $B$. 
Since $B$ does not have the cusp $c$, 
the hyperplane $\langle B\rangle$ must be an upper hemisphere which orthogonally 
intersects with $\langle A\rangle$ and $\langle A'\rangle$. 
Thus both $\langle A\rangle$ and $\langle A'\rangle $ share 
the north pole of $\langle B\rangle $ in $\mathbb{H}^3$. 
But $A$ and $A'$ are parallel. 
Thus $\langle A\rangle$ does not intersect $\langle A'\rangle$ in $\mathbb{H}^3$. 
Thus there are no $2$-dimensional faces which are adjacent to both $A$ and $A'$ other than $A''$ and $A'''$. 
By the same reason, there are no $2$-dimensional faces which are adjacent to both $A''$ and $A'''$ other than 
$A$ and $A'$. 
\end{proof}

By this sublemma, we know that $B_0$, $B_1$, $B_2$ and $B_3$ are different. 
For example, if $B_0$ and $B_1$ are the same face, then this face is adjacent to both $A_0$ and $A_2$, which contradicts Sublemma \ref{sublem:3-4}. 

From the above, $Q^3$ has at least eight $2$-dimensional faces. 
Suppose that $A_0$, $A_1$, $A_2$ and $A_3$ have $m$ vertices in total. 
Then, other than the edges terminating at the cusp $c$, faces $A_0$, $A_1$,
$A_2$ and $A_3$ have $m$ edges in total. 
Since each $B_i$ $(i=0, 1, 2, 3)$ is adjacent to two of $A_0$, $A_1$, $A_2$ and
$A_3$, there are eight edges shared by $B_i$ ($i=0, 1, 2, 3$) 
and one of $A_j$ $(j=0, 1, 2, 3)$. 
Then there are $(m-8)$ edges left when $m>8$. 
Take one of these edges, say $e$, and fix it. 
We may assume that $e$ belongs to $A_0$. 
Since any pair of faces can share only one edge, $e$
belongs to neither $B_0$ nor $B_3$. Therefore there is a new face 
$B_4$. Because of Sublemma \ref{sublem:3-4}, $B_4$ cannot be adjacent to $A_2$ which is
parallel to $A_0$. On the other hand, Sublemma \ref{sublem:3-3} tells us that $B_4$
cannot be adjacent to $A_1$ (resp.~$A_3$), since $B_4$ is 
different from $B_0$ (resp.~$B_3$) by our choice of $e$. 
Moreover, by Condition $(2)$, $B_4$ has a unique common edge with $A_0$. 
That is, if there is a $2$-dimensional face which is adjacent to $A_0$ but does not 
have the edge $e$, then this face is not $B_4$. 
By this observation, Sublemmata \ref{sublem:3-3}-\ref{sublem:3-4} and Condition $(2)$, 
the new compact $2$-dimensional faces which are adjacent to each of 
the $(m-8)$ edges above are different. 
Therefore, the number of faces 
in $Q^3$ must be greater than or equal to $m$, the number of vertices of the 
faces $A_0$, $A_1$, $A_2$ and $A_3$. 
The configuration of $2$-dimensional faces for the case $m=9$ 
is shown in Fig. \ref{fig:2}. 
\clearpage
\begin{figure}[htp]
\begin{center}
\includegraphics[width=10cm]{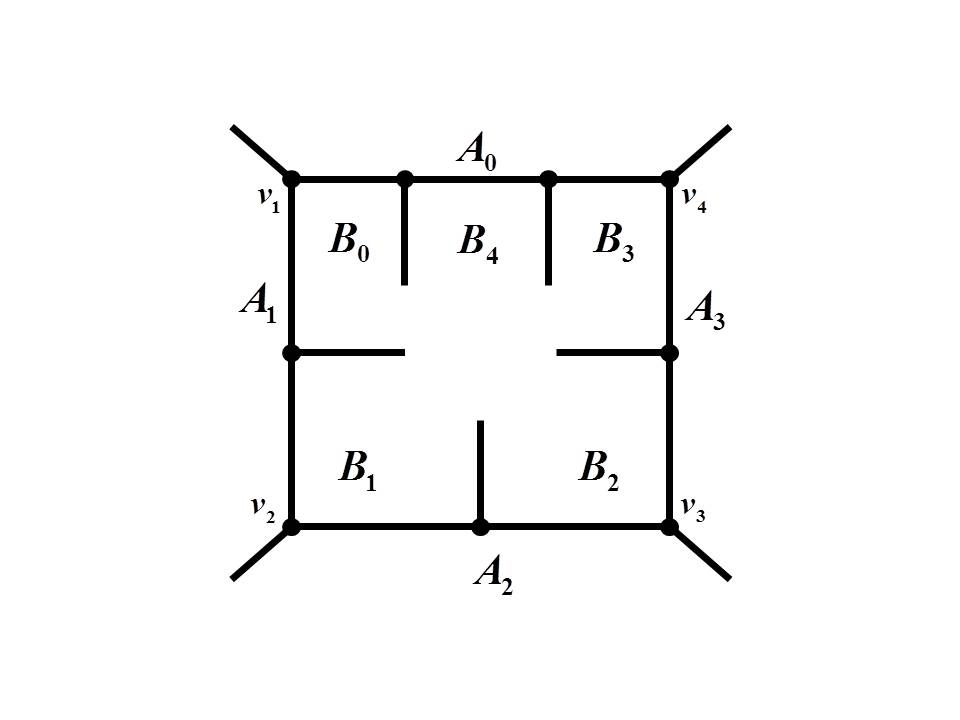}
\end{center}
\caption{The number of vertices of $A_0$, $A_1$, $A_2$ and $A_3$ is nine. }
\label{fig:2}
\end{figure}

\indent We add one more sublemma. 
\begin{sublem}\label{sublem:3-5}
\textit{The compact $2$-dimensional face $B_0$ $($resp. $B_1)$ is not adjacent to $B_2$ $($resp. $B_3)$. }
\end{sublem}
\begin{proof}
Suppose that $B_0$ $($resp. $B_1)$ is adjacent to $B_2$ $($resp. $B_3)$, 
then the four $2$-dimensional faces: $A_0$, $B_0$, $B_2$ and $A_3$ 
$($resp. $A_1$, $B_1$, $B_3$ and $A_0)$ are cyclically adjacent. 
These contradicts Condition $(5)$. 
Hence $B_0$ $($resp. $B_1)$ is not adjacent to $B_2$ $($resp. $B_3)$. 
\end{proof}

\indent In what follows, we present a proof of Lemma \ref{lem:3-1};  it is divided
into several cases according to the number of vertices of $A_0$, $A_1$,
$A_2$ and $A_3$. 

\subsection{The number of vertices of $A_0$, $A_1$, $A_2$ and $A_3$ is eight. }
\ 

\indent In this case, $A_0$, $A_1$, $A_2$ and $A_3$ are quadrangles. 
Thus, any compact $2$-dimensional face which is adjacent to $A_i$ $(i=0,1,2, 3)$ must be 
$B_0$, $B_1$, $B_2$ or $B_3$. For example, the compact $2$-dimensional faces which are adjacent to $A_0$ are only $B_0$ and $B_3$ since $A_0$ has exactly four edges. 
These four edges are $A_0\cap A_1$, $A_0\cap A_3$, $A_0\cap B_0$ and $A_0\cap B_3$. 
Since $A_0$ has one cusp and three vertices, the vertex of $A_0$ other than $v_1$ and $v_4$ must be $A_0\cap B_0\cap B_3$. Thus $B_0$ is adjacent to $B_1$. 
Similarly, we can show that $B_1$ $($resp. $B_2$, $B_3)$ is adjacent to $B_2$ $($resp. $B_3$, $B_0)$. 
On the other hand, Sublemma \ref{sublem:3-5} says that $B_0$ $($resp. $B_1)$ 
is not adjacent to $B_2$ $($resp. $B_3)$. 
Thus $B_0$, $B_1$, $B_2$ and $B_3$ are cyclically adjacent. 
But this does not occur by Condition $($5$)$. 
That is, the present case is impossible. 
\subsection{The number of vertices of $A_0$, $A_1$, $A_2$ and $A_3$ is nine. }
\ 

\indent In this case, we may assume that $A_0$ is a pentagon, 
and that $A_1$, $A_2$ and $A_3$ are quadrangles. 
We denote by $B_4$ the new compact $2$-dimensional face 
which is adjacent to $A_0$ other than $B_0$ and $B_3$. 
Note that $B_4$ is adjacent to $B_0$ and $B_3$. 
If $B_4$ is adjacent to $B_1$ $($resp. $B_2)$, then $A_0$, $B_4$, $B_1$ $($resp. $B_2)$ and 
$A_1$ $($resp. $B_2)$ are cyclically adjacent. But this does not satisfy Condition $(5)$. 
Thus $B_4$ is not adjacent to $B_1$ or $B_2$. 
Since $B_4$ has at least five edges, 
it is adjacent to at least two $2$-dimensional faces besides $A_0$, $B_0$ and $B_3$. 
Hence $Q^3$ has at least eleven $2$-dimensional faces. 

\indent Assume that $Q^3$ has exactly eleven $2$-dimensional faces. 
In this case, $B_4$ is adjacent to exactly five $2$-dimensional faces: $A_0$, $B_0$, $B_3$ 
and two other $2$-dimensional faces which are compact. 
One of the last two compact $2$-dimensional faces is adjacent to $B_0$, and the other is adjacent to $B_3$. 
We denote by $F$ the first one, and denote by $F'$ the second one; 
$F$ (resp. $F'$) is adjacent to both $B_4$ and $B_0$ (resp. $B_4$ and $B_3$). 
If $B_0$ is adjacent to $F'$, then $B_0$, $F'$ and $B_4$ are pairwise adjacent. 
Thus, in this case, these faces must share a vertex by Condition $($4$)$. 
But the endpoints of the edge $B_0\cap B_4$ are $A_0\cap B_0\cap B_4$ and $B_0\cap B_4\cap F$. 
That is to say, $B_0$ is not adjacent to $F'$. 
Since $A_0$ has five edges, $B_0$ is not adjacent to $B_3$ by Condition $($4$)$. 
Since $B_0$ is adjacent to both $A_0$ and $A_1$, by Sublemma \ref{sublem:3-4}, 
it is adajacent to neither $A_2$ nor $A_3$. 
Moreover, by Sublemma \ref{sublem:3-5}, $B_0$ is not adjacent to $B_2$. 
Thus $B_0$ can be adjacent only to $A_0$, $A_1$, $B_1$, $B_4$ and $F$. 
On the other hand, by the same reason, $B_3$ is adjacent only to $A_0$, $A_3$, $B_2$, $B_4$ and $F'$. 
Thus $F$ must be adjacent to $B_0$, $B_1$, $B_2$, $B_4$ and $F'$ 
to satisfy Condition $(1)$. 
But in this case $F$, $B_2$, $B_3$ and $B_4$ are cyclically adjacent. 
This is a contradiction to Condition $(5)$. 
Hence $Q^3$ has more than eleven $2$-dimensional faces. 

\indent Now we assume that $Q^3$ has exactly twelve $2$-dimensional faces. 
Then there are exactly eight compact $2$-dimensional faces of $Q^3$. 
We denote by $F$, $F'$ and $F''$ the compact $2$-dimensional faces which are different from 
$B_i$ $(i=0, 1, 2, 3, 4)$. 
We can choose $F$ and $F'$ as above. 
If $F$ is adjacent to $B_2$, then $F$, $B_2$, $B_3$ and $B_4$ are cyclically adjacent. 
And if $F$ is adjacent to $B_3$, 
then $F$, $B_3$ and $B_4$ are pairwise adjacent but do not share a vertex 
because the endpoints of the edge $B_3\cap B_4$ are 
$B_3\cap B_4\cap A_0$ and $B_3\cap B_4\cap F'$. 
Thus, by Conditions $($4-5$)$, 
$F$ can be adjacent to neither $B_2$ nor $B_3$. 

\indent Suppose that $B_0$ is adjacent to $F''$. 
Now we consider the $2$-dimensional faces which $F''$ can be adjacent to. 
If $F$ is adjacent to $B_1$, then $B_0$, $B_1$ and $F$ are pairwise adjacent. 
And then they must have a common vertex by Condition $($4$)$. 
It means that one of the endpoints of the edge $B_0\cap F$ is $B_0\cap B_1\cap F$ and 
the other is $B_0\cap B_4\cap F$ by Condition $($2$)$. 
But in this case $B_0$ has only five edges: 
$A_0\cap B_0$, $A_1\cap B_0$, $B_0\cap B_1$, $B_0\cap B_4$, $B_0\cap F$. 
That is, $B_0$ cannot be adjacent to $F''$. This contradicts our assumption. 
Thus $F$ is not adjacent to $B_1$. 
Thus $F$ can be adjacent only to $B_0$, $B_4$, $F'$ and $F''$. 
But this violates Condition $($1$)$. 
In the end, $B_0$ cannot be adjacent to $F''$. 
Analogously, we can conclude that $B_3$ cannot be adjacent to $F''$. 
Thus $F''$ must be adjacent to $B_1$, $B_2$, $B_4$, $F$ and $F'$. 
In this case, $B_0$, $B_4$, $F''$, $B_1$ are cyclically adjacent. 
This is a contradiction to Condition $(5)$. 

\indent Hence, in this case, $Q^3$ must have more than twelve $2$-dimensional faces. 
\begin{figure}[htp]
\begin{center}
\includegraphics[width=8cm]{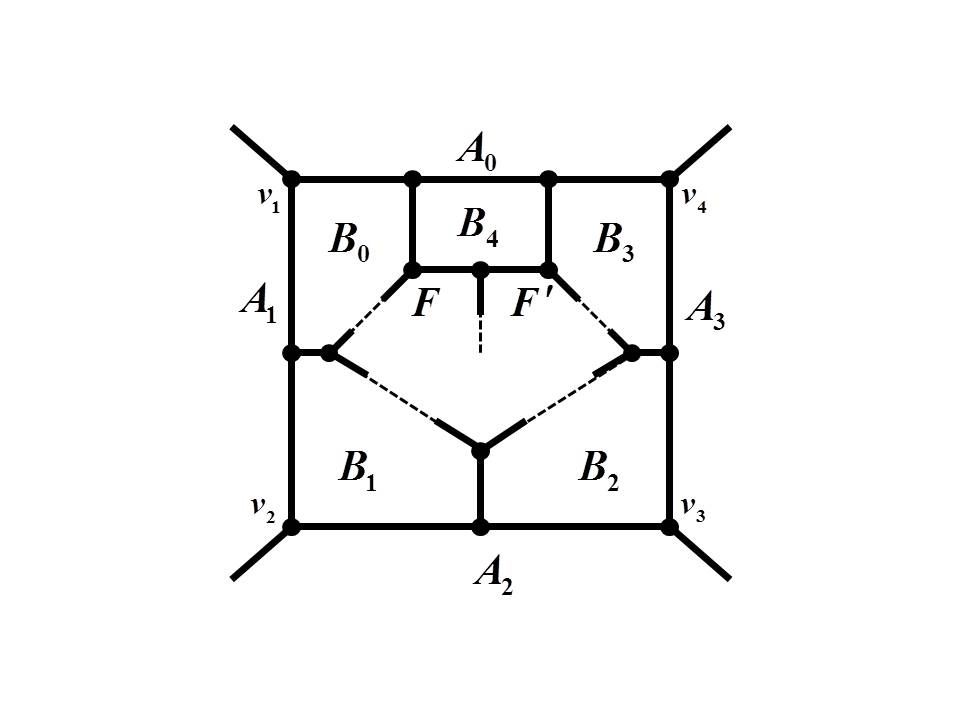}
\caption{$a_2(Q^3)=11$ in the case 3.2.}
\label{fig:2-1}
\includegraphics[width=8cm]{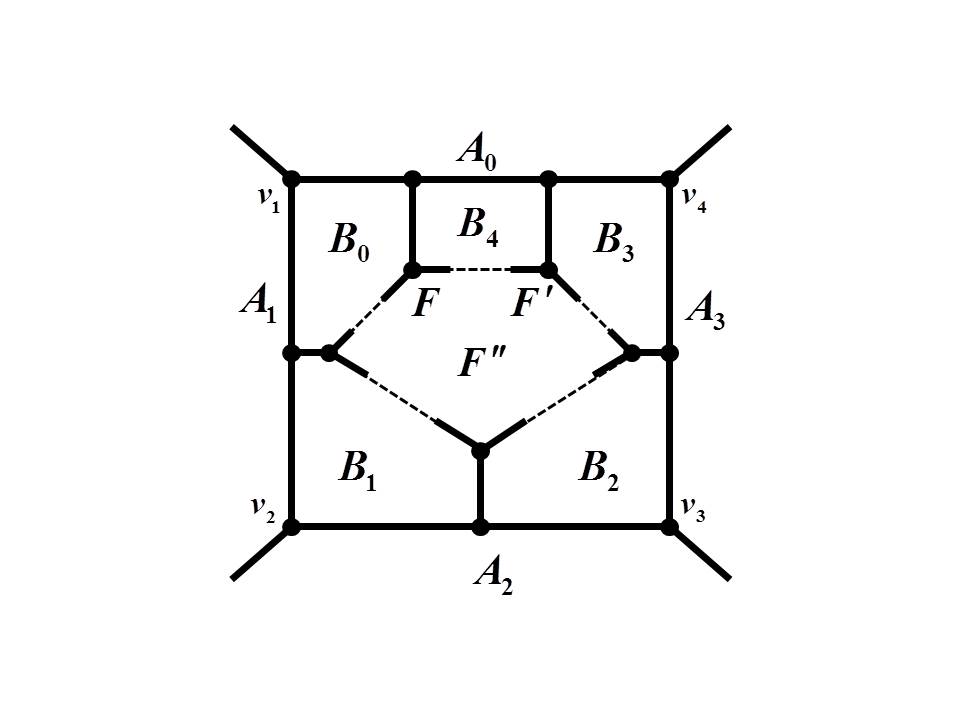}
\end{center}
\caption{$a_2(Q^3)=12$ in the case 3.2.}
\label{fig:2-2}
\end{figure}
\clearpage

\subsection{The number of vertices of $A_0$, $A_1$, $A_2$ and $A_3$ is ten. }
\ 

\indent Under this assumption, we may assume that one of the following three cases occurs: 

\indent $($\ref{sec:2-3-1}$)$ the face $A_0$ is a hexagon, and the faces $A_1$, $A_2$ and $A_3$ are quadrangles $($Fig. \ref{fig:3-1}$)$, 

\indent $($\ref{sec:2-3-2}$)$ the faces $A_0$ and $A_1$ are pentagons, 
and the faces $A_2$ and $A_3$ are quadrangles 
$($Fig. \ref{fig:3-2}$)$, 

\indent $($\ref{sec:2-3-3}$)$ the faces $A_0$ and $A_2$ are pentagons, 
and the faces $A_1$ and $A_3$ are quadrangles $($Fig. \ref{fig:3-3}$)$. 
\subsubsection{The face $A_0$ is a hexagon, and the faces $A_1$, $A_2$ and $A_3$ are quadrangles.}
\label{sec:2-3-1}
 
\indent We denoted by $B_4$ the compact $2$-dimensional face which is adjacent to both $A_0$ and $B_0$. 
Let $B_5$ be the compact $2$-dimensional face which is adjacent to $A_0$, $B_3$ and $B_4$. 
If $B_4$ is adjacent to $B_1$ $($resp. $B_2)$, then the four $2$-dimensional faces 
$A_0$, $A_1$, $B_1$ and $B_4$ $($resp. $A_0$, $A_3$, $B_2$ and $B_4)$ are 
cyclically adjacent. But this contradicts Condition $(5)$. 
That is to say, $B_4$ is adjacent to neither $B_1$, nor $B_2$. 
In addition, since $A_0$, $B_3$ and $B_4$ cannot be pairwise adjacent 
by Condition $(4)$, $B_4$ is not adjacent to $B_3$. 
Thus there are at least two compact $2$-dimensional faces which are adjacent to $B_4$, 
other than $B_0$ and $B_5$. 
We denote these faces by $F$ and $F'$. We suppose that $a_2(Q^3)=12$. 
Then the compact $2$-dimensional faces are only $B_i$ $(i=0, 1, 2, 3, 4, 5)$, $F$ and $F'$. 
Note that $B_5$ also has to be adjacent to $F$ and $F'$. 
Thus $F$, $F'$ and $B_4$ are pairwise adjacent, and $F$, $F'$ and $B_5$ are also pairwise adjacent. 
This means that one of the endpoints of the edge $F\cap F'$ is in $B_4$,  the other is in $B_5$. 
But in this case, three compact $2$-dimensional faces $B_4$, $B_5$ and $F$ $($or $F')$ 
must be pairwise adjacent but do not share a vertex. 
This is impossible by Condition $(4)$. Hence $a_2(Q^3)>12$. 
\clearpage
\begin{figure}[htp]
\begin{center}
\includegraphics[width=8cm]{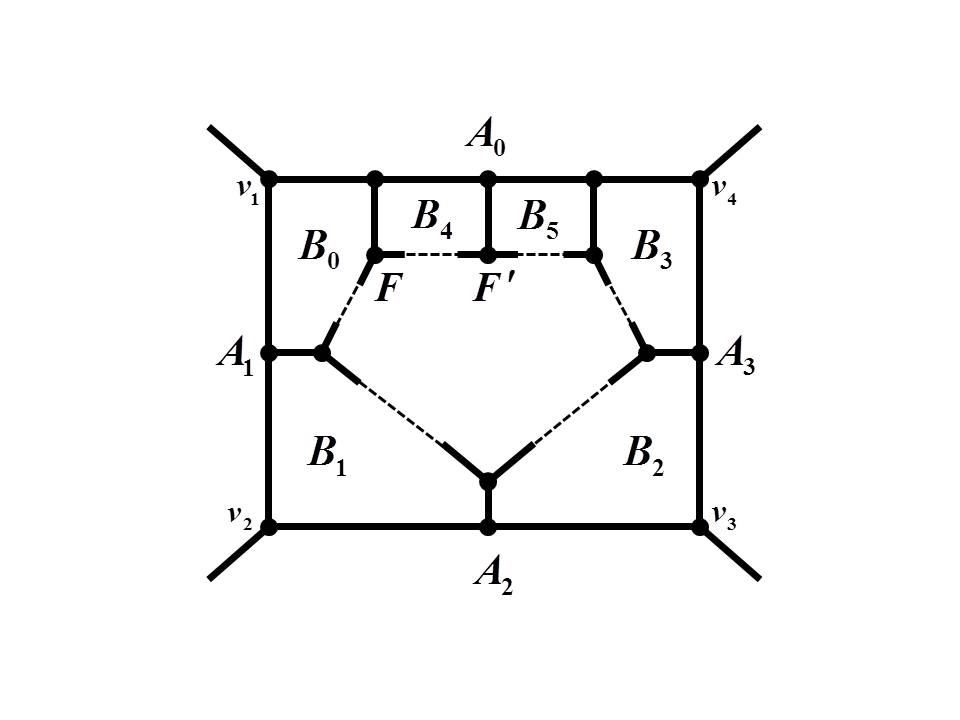}
\end{center}
\caption{$A_0$ is a hexagon, and $A_1$, $A_2$ and $A_3$ are quadrangle. }
\label{fig:3-1}
\end{figure}

\subsubsection{The faces $A_0$ and $A_1$ are pentagons, and the faces $A_2$ and $A_3$ are quadrangles. }
\label{sec:2-3-2}
\indent Let $B_6$ be the compact $2$-dimensional face which is adjacent to $A_1$, $B_0$ and $B_1$. 
Denote by $F$ the compact $2$-dimensional face which is adjacent to $B_0$ and $B_4$, 
and denote by $F'$ the compact $2$-dimensional face which is adjacent to $B_3$ and $B_4$. 
Suppose that $Q^3$ has exactly twelve $2$-dimensional faces. 
By analogy to the case $($\ref{sec:2-3-1}$)$, 
we can show that $B_4$ and $B_6$ are adjacent to $F$ and $F'$. 
But in this case there exist four $2$-dimensional faces $B_4$, $B_0$, $B_6$ and $F$ $($or $F')$, 
which are cyclically adjacent. 
This is a contradiction to Condition $(5)$. 
Hence $Q^3$ has more than twelve $2$-dimensional faces. 
\clearpage
\begin{figure}[htp]
\begin{center}
\includegraphics[width=10cm]{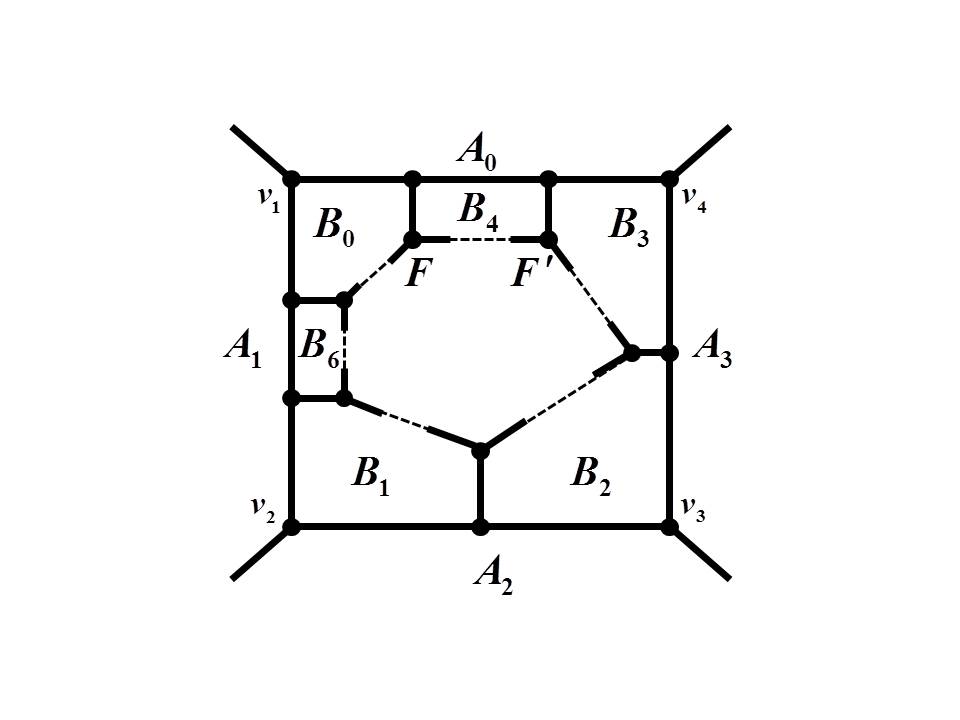}
\end{center}
\caption{$A_0$ and $A_1$ are pentagons, and $A_2$ and $A_3$ are quadrangles. }
\label{fig:3-2}
\end{figure}

\subsubsection{The faces $A_0$ and $A_2$ are pentagons, and the faces $A_1$ and $A_3$ are quadrangles. }
\label{sec:2-3-3}
\indent Let $B_7$ be the compact $2$-dimensional face which is adjacent to $A_2$, $B_1$ and $B_2$. 
The facets denoted by $F$ and $F'$ are as in the case $($\ref{sec:2-3-2}$)$. 
By Conditions $($4-5$)$, $F$ is different from $B_1$, $B_2$, $B_3$, $B_7$ and $F'$. 
By the same reason, $F'$ is different from $B_0$, $B_1$, $B_2$ and $B_7$. 
Thus $Q^3$ has at least twelve $2$-dimensional faces. 
Note that if $Q^3$ has exactly twelve $2$-dimensional faces, 
then the combinatorial structure of $Q^3$ is exactly as shown in Fig. \ref{fig:face12}. 
By Andreev's theorem $($Theorem \ref{thm:And}$)$, 
it is easy to prove that there exists such a polyhedron in $\mathbb{H}^3$. 
\clearpage
\begin{figure}[htp]
\begin{center}
\includegraphics[width=10cm]{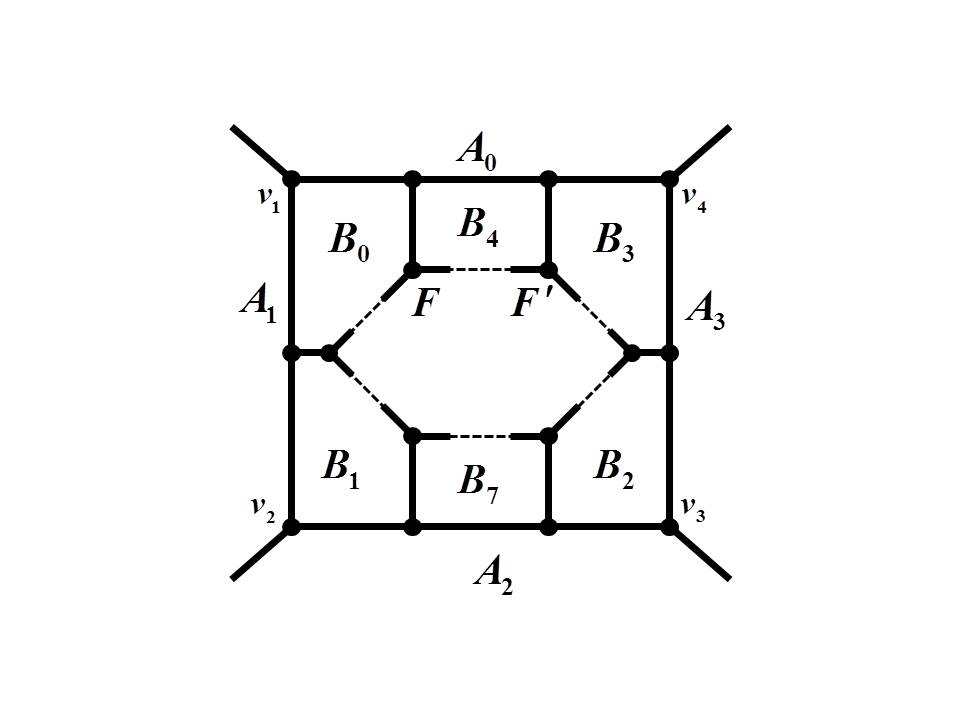}
\end{center}
\caption{$A_0$ and $A_2$ are pentagons, and $A_1$ and $A_3$ are quadrangles. }
\label{fig:3-3}
\begin{center}
\includegraphics[width=10cm]{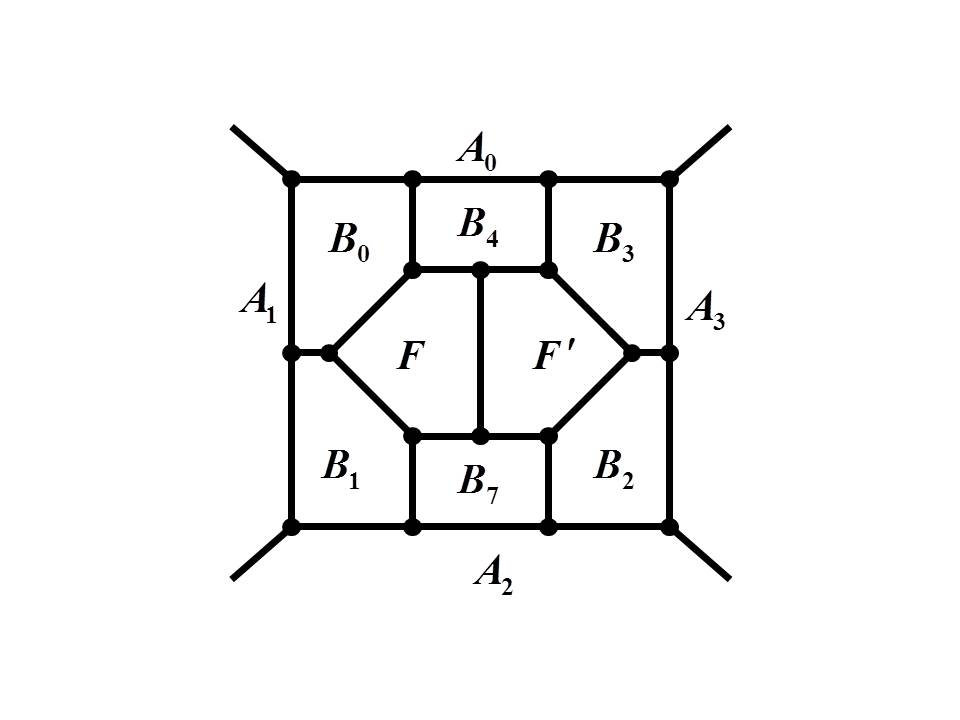}
\end{center}
\caption{$Q^3$ which has exactly one cusp and twelve $2$-dimensional faces.}
\label{fig:face12}
\end{figure}
\clearpage

\subsection{The number of vertices in $A_0$, $A_1$, $A_2$ and $A_3$ is eleven. }
\ 

\indent In this case, there are seven compact $2$-dimensional faces 
which are adjacent to either $A_0$, $A_1$, $A_2$ or $A_3$. 
And there are some other compact $2$-dimensional faces 
which are adjacent to those compact faces. 
Thus $Q^3$ has at least twelve $2$-dimensional faces. 
If $Q^3$ has exactly twelve $2$-dimensional faces, 
then there is only one compact $2$-dimensional face which is adjacent to the compact faces 
which are adjacent to either $A_0$, $A_1$, $A_2$ or $A_3$. 
But if there is only one such a face, 
then there must be some compact $2$-dimensional faces 
which are adjacent to either $A_0$, $A_1$, $A_2$ or $A_3$, 
and which are either triangles or quadrangles. 
But they do not meet Condition $($1$)$. 
Thus $Q^3$ has strictly more than twelve $2$-dimensional faces. 
\subsection{The number of vertices in $A_0$, $A_1$, $A_2$ and $A_3$ is twelve. }
\ 

\indent In this case, there are eight compact $2$-dimensional faces 
which are adjacent to either $A_0$, $A_1$, $A_2$ or $A_3$. 
And there are some other compact faces which are adjacent to those compact faces. 
Thus $Q^3$ has more than twelve $2$-dimensional faces. 
\subsection{The number of vertices in $A_0$, $A_1$, $A_2$ and $A_3$ is greater than twelve. }
\ 

\indent In this case, there are more than nine compact $2$-dimensional faces 
which are adjacent to either $A_0$, $A_1$, $A_2$ or $A_3$. 
Thus $Q^3$ has more than twelve $2$-dimensional faces. 

\subsection{Conclusion of the proof of Lemma \ref{lem:3-1} and a corollary.}
\ 

\indent By the cases 3.1-3.6, we have proved Lemma \ref{lem:3-1}. 

\indent By Proposition \ref{pro:2-2} and Lemma \ref{lem:3-1}, we obtain the following corollary. 
\begin{cor}\label{cor:2-1} 
If $c(Q^3)\leq 1$, then $a_2(Q^3)\geq 12$. 
\end{cor}
\begin{rem}
If $c(Q^3)=0$ and $a_2Q^3)=12$, then we have the compact right-angled dodecahedron. 
If $c(Q^3)=1$ and $a_2(Q^3)=12$, then we have a polyhedron in Fig. \ref{fig:face12}, 
which arises by contracting an edge of the dodecahedron above, 
as described in {\rm \cite{ref:ko}}. 
\end{rem}

\section{Proof of Main Theorem: $n=6$}
\indent Assume that $Q^6$ has exactly one cusp. 
Then any $3$-dimensional face of $Q^6$ has at most one cusp. 
Thus any $3$-dimensional face of $Q^6$ has at least twelve 
$2$-dimensional faces by Corollary \ref{cor:2-1}. 
Thus we obtain $a_3^2(Q^6)\geq 12$. 
But by Nikulin's inequality $($Theorem \ref{thm:2-4}$)$, 
we obtain the opposite inequality $a_3^2(Q^6)<12$. 
Hence $Q^6$ has more than one cusp. 

\indent Now we assume that $Q^6$ has exactly two cusps. 
If each of the $3$-dimensional faces of $Q^6$ has at most one cusp, 
then each $3$-dimensional faces of $Q^6$ has more than twelve $2$-dimensional faces. 
Thus we obtain $a_3^2(Q^6)\geq 12$. 
But by Theorem \ref{thm:2-4}, we obtain the inequality $a_3^2(Q^6)<12$. 
Hence there is a $3$-dimensional face which has two cusps. Denote this face by $G$. 
We denote one cusp of $G$ by $c$, and the other cusp by $c'$.  
There are three possibilities: 

\indent $($\ref{sec:3-1-1}$)$ there are no $2$-dimensional faces of $G$ which have two cusps, 

\indent $($\ref{sec:3-1-2}$)$ there is only one $2$-dimensional face which has two cusps, 

\indent $($\ref{sec:3-1-3}$)$ there are two $2$-dimensional faces which have two cusps. 

\indent Note that there is one edge which starts at $c$, and terminates at $c'$ in the case $($\ref{sec:3-1-3}$)$. 

\indent Define the $2$-dimensional faces $A_0$, $A_1$, $A_2$ and $A_3$ as before. 
The case $($\ref{sec:3-1-1}$)$ $($resp. $($\ref{sec:3-1-2}$)$, $($\ref{sec:3-1-3}$))$ 
is depicted in Fig. 9 $($resp. 10, 11$)$. 
A circle in these figures represents the cusp $c'$. 
From now on, we examine each of the above cases. 
\subsection{There are no $2$-dimensional faces of $G$ which have two cusps. }\label{sec:3-1-1}
\ 

\indent The number of $2$-dimensional faces of $G$ sharing $c$ is four, 
and the number of $2$-dimensional face sharing $c'$ is also four. 
There are no $2$-dimensional faces which have both cusps. 
Thus $a_2(G)\geq 8$. 
\begin{figure}[htp]
\begin{center}
\includegraphics[width=9cm]{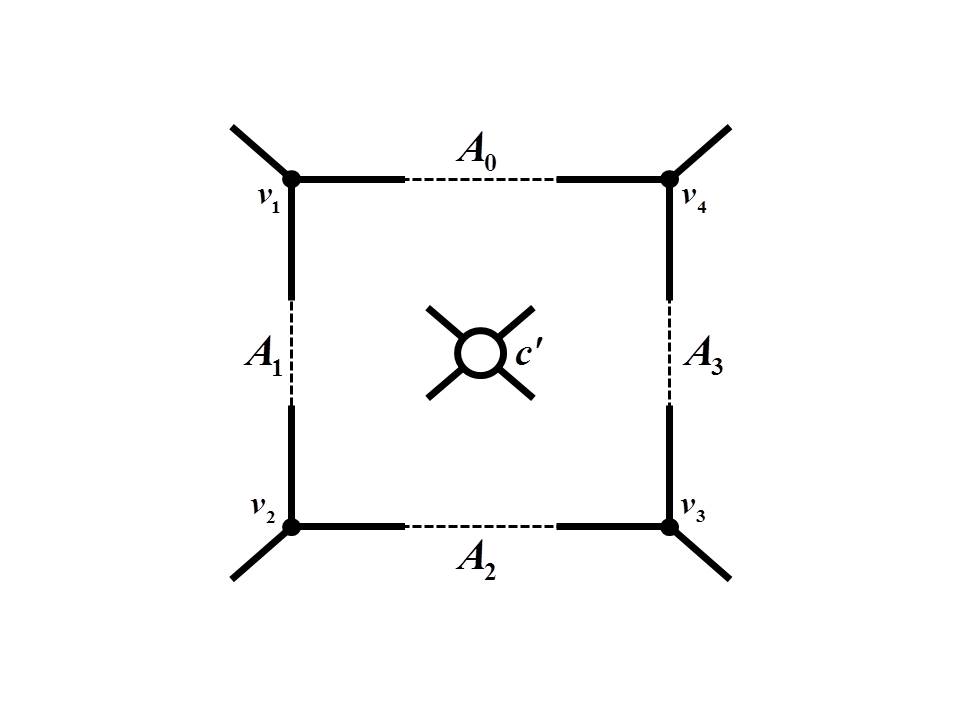}
\end{center}
\caption{The case $($\ref{sec:3-1-1}$)$}
\label{fig:4-1}
\end{figure}

\subsection{There is only one $2$-dimensional face which has two cusps. }\label{sec:3-1-2}
\ 

\indent We may assume that $A_0$ has two cusps. 
Let $B_0$ be the $2$-dimensional face which is adjacent to $A_0$ and $A_1$, 
and let $B_1$ be the $2$-dimensional face which is adjacent to $A_1$ and $A_2$. 
Let $B_2$ be the $2$-dimensional face which is adjacent to $A_2$ and $A_3$, 
and let $B_3$ be the $2$-dimensional face which is adjacent to $A_3$ and $A_0$. 
Because $A_0$ has $c'$, 
there is one $2$-dimensional face which is parallel to $A_0$ and sharing the cusp $c'$ with $A_0$. 
We denote this face by $B_4$. 
Since both $B_0$ and $B_3$ are adjacent to $A_0$, they cannot coincide with $B_4$. 
If $G$ has exactly eight $2$-dimensional faces, 
then $B_4$ must coincide with either $B_1$ or $B_2$. 
Assume that $B_4$ is $B_1$. 
Because $B_0$ is adjacent to $A_0$ and $A_1$, it is adjacent to neither $A_2$ nor $A_3$. 
If $B_0$ is adjacent to $B_2$, then $B_0$, $B_2$, $A_3$ and $A_0$ are cyclically adjacent. 
But this does not occur because of Theorem \ref{thm:And} $(e)$. Thus $B_0$ is not adjacent to $B_2$. 
Eventually, $B_0$ can be adjacent only to $A_0$, $A_1$ and $B_4$ in this case. 
That is to say, $B_0$ must be a triangle. 
But $B_0$ has more than three edges because of Proposition \ref{pro:2-2}. 
This means that $B_4$ cannot coincide with $B_1$. 
By the same reason, $B_4$ cannot coincide with $B_2$. Hence $a_2(G)\geq 9$. 
\begin{figure}[htp]
\begin{center}
\includegraphics[width=10cm]{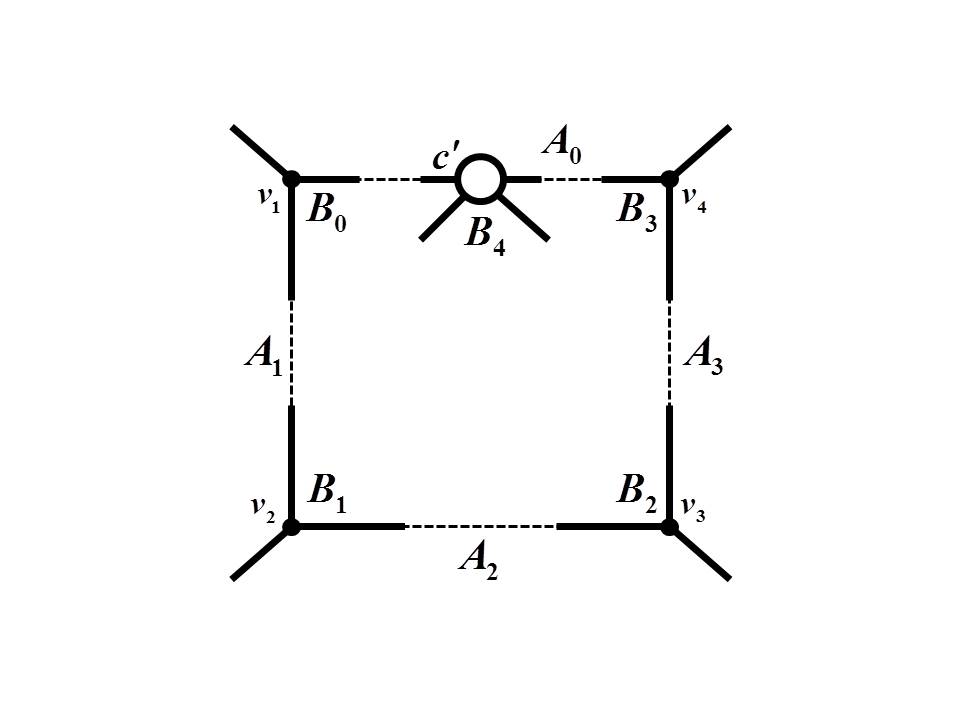}
\end{center}
\caption{The case $($\ref{sec:3-1-2}$)$}
\label{fig:4-2}
\end{figure}
\clearpage
\subsection{There are two $2$-dimensional faces which have two cusps. }\label{sec:3-1-3}
\ 

Let $E_1$ and $E_2$ be the $2$-dimensional faces which share the cusp $c'$, other than $A_0$ and $A_1$. 
Note that exactly four $2$-dimensional faces $A_0$, $A_1$, $E_1$ and $E_2$ have $c'$. 
We may assume that $E_1$ (resp. $E_2$) and $A_0$ (resp. $A_1$) are parallel. 
Denote by $B_2$ the $2$-dimensional face which is adjacent to $A_2$ and $A_3$. 
By Theorem \ref{thm:And} $(c)$, $B_2$ can be adjacent to neither $A_0$, nor $A_1$. 
That is to say, $B_2$ is neither $E_1$ nor $E_2$. 
Thus $B_2$ is compact. Then $B_2$ has to be adjacent to at least five $2$-dimensional faces of $G$. 
Then there exists a compact $2$-dimensional face 
which is adjacent to $B_2$. 
Denote this face by $H$. 

\indent Assume that $a_2(G)=8$. 
In this case, the $2$-dimensional faces of $G$ are only 
$A_0$, $A_1$, $A_2$, $A_3$, $E_1$, $E_2$, $B_2$ and $H$. 
If $H$ is adjacent to both $A_2$ and $A_3$, then $H$ must coincide with $B_2$. 
Thus $H$ can be adjacent only to one of $A_2$ and $A_3$. 
By Sublemma \ref{sublem:3-4}, $H$ can be adjacent only to one of $A_0$ and $E_1$, 
and be adjacent to one of $A_1$ and $E_2$. 
Thus $H$ can be adjacent only to at most four $2$-dimensional faces of $G$. 
This is a contradiction to Condition $($1$)$ because $H$ is compact. 
Hence $a_2(G)\not= 8$. 

\indent Assume that $a_2(H)=9$. In this case, there is another compact $2$-dimensional face of $G$. 
Denote this face by $J$. 

\indent Suppose that $H$ is adjacent to $A_0$. Then $H$ is not adjacent to $A_1$ 
because $A_0$, $A_1$ and $H$ do not have a common vertex. 
Moreover, $H$ is adjacent to neither $A_2$ nor $E_1$ by Sublemma \ref{sublem:3-4}. 
Thus, $H$ must be adjacent to $A_0$, $A_3$, $E_2$, $B_2$ and $J$ 
because it has at least five edges. 
Thus the endpoints of the edge $A_3\cap H$ are $A_0\cap A_3\cap H$ and $A_3\cap B_2\cap H$. 
That is to say, $A_3$ is adjacent only to four $2$-dimensional faces $A_0$, $A_2$, $B_2$ and $H$. 
Because $A_0$, $E_2$ and $H$ are pairwise adjacent, $A_0$ is adjacent only to four $2$-dimensional faces 
$A_1$, $A_3$, $E_2$ and $H$. 
Thus $J$ can be adjacent to $A_1$, $A_2$, $E_1$, $E_2$, $B_2$ and $H$. 
But, since $A_1$ and $E_2$ are parallel, $J$ can be adjacent only to one of them. 
Thus $J$ must be adjacent to $A_2$, $E_1$, $B_2$ and $H$ to satisfy Condition $($1$)$. 
In this case, $B_2$, $H$ and $J$ are pairwise adjacent, so they have a common vertex. 
Thus the four $2$-dimensional faces $A_2$, $A_3$, $H$ and $J$ are cyclically adjacent but share 
neither a vertex nor a cusp 
because the endpoints of the edge $A_2\cap A_3$ are the cusp $c$ and the vertex $A_2\cap A_3\cap B_2$. 
This is a contradiction to Theorem \ref{thm:And} $(e)$. 
Thus $H$ is not adjacent to $A_0$. 

\indent Analogously, we can prove that $H$ is not adjacent to $A_1$. 
Thus $H$ is adjacent to $E_1$, $E_2$, $B_2$, $J$ and one of $A_2$ and $A_3$. 
Since one of the endpoints of the edge $E_1\cap E_2$ is the cusp $c'$, 
$H$ is the only $2$-dimensional face which is adjacent to both $E_1$ and $E_2$. 
Thus $J$ can be adjacent only to one of $E_1$ and $E_2$. 
In this manner, we consider the $2$-dimensional face $J$. 
By Sublemma \ref{sublem:3-4}, $J$ can be adjacent only to one of $A_0$ and $A_2$, 
and only to one of $A_1$ and $A_3$. 
Because there is only one $2$-dimensional face which is adjacent to both $A_2$ and $A_3$, 
$J$ cannot be adjacent to both $A_2$ and $A_3$. 
The common edge of $A_0$ and $A_1$ starts from one cusp and terminates at the other cusp. 
Thus, by Theorem \ref{thm:And} $(c)$, 
if there exists a $2$-dimensional face which is adjacent to both $A_0$ and $A_1$, 
then it must share a cusp with $A_0$ and $A_1$. 
But the $2$-dimensional faces which share a cusp with $A_0$ and $A_1$ are nothing but 
the $2$-dimensional faces $A_2$, $A_3$, $E_1$ and $E_2$. 
Altogether, there is no $2$-dimensional face which is adjacent to both $A_0$ and $A_1$. 
Thus there are only two possibilities. One is that $J$ can be adjacent only to $A_0$, $A_3$, $E_2$ and $H$, 
and the other is that $J$ can be adjacent only to $A_1$, $A_2$, $E_1$ and $H$. 
But both cases do not meet Condition $($1$)$. 
Thus $a_2(G)\not= 9$, and hence  $a_2(G)\geq 10$. 

\begin{figure}[htp]
\begin{center}
\includegraphics[width=10cm]{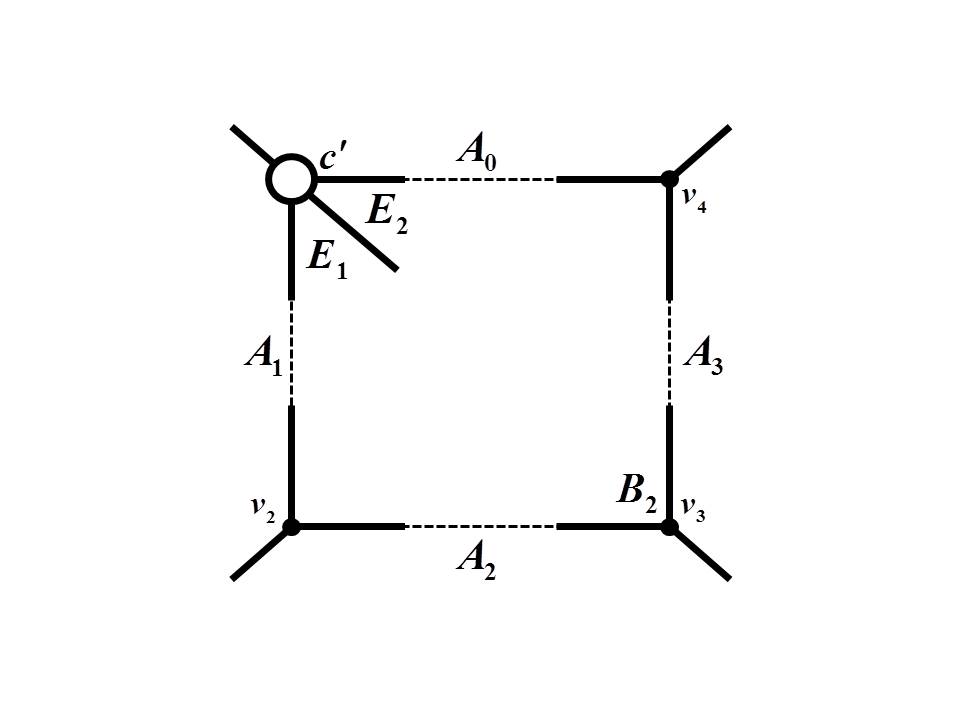}
\end{center}
\caption{The case $($\ref{sec:3-1-3}$)$}
\label{fig:4-3}
\end{figure}

\subsection{An estimate for $a_3^2(Q^6)$ and the conclusion of the proof of the Main Theorem.}
\ 

\indent Denote one cusp of $Q^6$ by $c$. 
There are exactly ten $5$-dimensional faces of $Q^6$ which share $c$. 
Denote these $5$-dimensional faces by $P_i$ $(i=1$, $\cdots $, $10)$. 
Assume that $P_i$ and $P_{11-i}$ $(1\leq i\leq 10)$ are parallel. 
It is easy to see that any $k$-dimensional face having $c$ 
is the intersection of $(6-k)$ $5$-dimensional faces. 
Thus there are eighty $3$-dimensional faces which have $c$. 
Assume that the $3$-dimensional face $P_1\cap P_2\cap P_3$ has two cusps. 

\begin{lem}\label{lem:4-1} 
If there is a $3$-dimensional face which has two cusps, 
then this face and $P_1\cap P_2\cap P_3 $ must satisfy one of the following cases: 

$(1)$ They are the same $3$-dimensional face. 

$(2)$ They have one common $2$-dimensional face 
which has two cusps but do not have an edge which joins these two cusps. 

$(3)$ They have one common edge which starts at one cusp and terminates at the other cusp. 
\end{lem}

\begin{proof}
Denote the $3$-dimensional face which has two cusps by $P_i\cap P_j\cap P_k$ 
$(i$, $j$ and $k$ are different$)$. 
It is sufficient to prove that $P_i\cap P_j\cap P_k$ 
is a $3$-dimensional face of either $P_1$, $P_2$ or $P_3$. 
If this $3$-dimensional face is a $3$-dimensional face of either $P_8$, $P_9$ or $P_{10}$, 
then it is parallel to $P_1\cap P_2\cap P_3$. 
But in this case, $P_i\cap P_j\cap P_k$ must have exactly one cusp. 
Thus $P_i\cap P_j\cap P_k$ is not a $3$-dimensional face of either $P_8$, $P_9$ or $P_{10}$. 
Because $P_5$ and $P_6$ are parallel, if $P_i\cap P_j\cap P_k$ is 
a $3$-dimensional face of either $P_5$ or $P_6$, then it cannot be a $3$-dimensional face of the other one. 
Thus $P_i\cap P_j\cap P_k$ is a $3$-dimensional face of either $P_1$, $P_2$ or $P_3$. 
\end{proof}

To prove the main theorem, we need to prove the following lemma. 
\begin{lem}\label{lem:4-2}
Assume that three $3$-dimensional faces of $Q^6$ have a common cusp, and they do not share the other one. 
If these $3$-dimensional faces are pairwise adjacent, 
then one of them has more than twelve $2$-dimensional faces. 
\end{lem}

\begin{proof}
Let $J_1$, $J_2$ and $J_3$ be the three $3$-dimensional faces 
which have a common cusp, and do not share the other one. 
By the assumption, $J_1$, $J_2$ and $J_3$ are pairwise adjacent. 
By Lemma \ref{lem:3-1}, each of these faces has at least twelve $2$-dimensional faces. 
Assume that one of these $3$-dimensional faces 
has exactly twelve $2$-dimensional faces. 
In this case, the intersection $J_1\cap J_2$ is a quadrangle or a pentagon. 

\indent Assume that $J_1\cap J_2$ is a quadrangle. 
Because the intersection $J_1\cap J_3$ is a $2$-dimensional face of $J_1$, 
and is adjacent to $J_1\cap J_2$, $J_1\cap J_3$ must be a pentagon by Lemma \ref{lem:3-1}. 
The $2$-dimensional face $J_2\cap J_3$ must be a pentagon 
because the intersection $J_2\cap J_3$ is a $2$-dimensional face of $J_2$, 
and it is adjacent to $J_1\cap J_2$. 
The intersections $J_1\cap J_3$ and $J_2\cap J_3$ are a $2$-dimensional faces of $J_3$. 
Note that these two faces are adjacent. 
Thus $J_3$ must have more than twelve $2$-dimensional faces by Lemma \ref{lem:3-1}. 
But $J_3$ has only twelve $2$-dimensional faces. Hence $J_1\cap J_2$ is not a quadrangle. 
Thus $J_1\cap J_2$ is a pentagon. 
But as we have seen above, $J_1\cap J_3$ and $J_2\cap J_3$ are quadrangles, 
and thus $J_3$ must have more than twelve $2$-dimensional faces. 
Hence either $J_1$, $J_2$ or $J_3$ must have more than twelve $2$-dimensional faces. 
\end{proof}

\indent We consider the case $($\ref{sec:3-1-1}$)$ again. 
By Lemma \ref{lem:4-1}, there is exactly one $3$-dimensional face which has two cusps. 
We may assume that this $3$-dimensional face is the intersection of three $5$-dimensional faces 
$P_1$, $P_2$ and $P_3$. This $3$-dimensional face has more than seven $2$-dimensional faces 
by Proposition \ref{pro:2-2}. 
Moreover, by Corollary \ref{cor:2-1}, this face is the only $3$-dimensional face 
which has less than twelve $2$-dimensional faces. 
By Theorem \ref{thm:2-4}, we obtain the following inequality; 
\begin{equation}\label{eq:1} 
a_3^2(Q^6)<12. 
\end{equation}
If there are more than three $3$-dimensional faces which have more than twelve $2$-dimensional faces each, 
then $Q^6$  cannot satisfy $($\ref{eq:1}$)$. 
By Lemma \ref{lem:4-2}, one of the $3$-dimensional faces 
$P_1\cap P_2\cap P_4$, $P_1\cap P_2\cap P_5$ and $P_1\cap P_2\cap P_8$ 
has more than twelve $2$-dimensional faces. 
Similarly, one of the $3$-dimensional faces 
$P_1\cap P_4\cap P_5$, $P_2\cap P_4\cap P_5$ and $P_4\cap P_5\cap P_8$ 
has more than twelve $2$-dimensional faces. 
And one of the $3$-dimensional faces 
$P_1\cap P_4\cap P_6$, $P_2\cap P_4\cap P_6$ and $P_4\cap P_6\cap P_8$ 
has more than twelve $2$-dimensional faces. 
In addition, one of the $3$-dimensional faces 
$P_1\cap P_3\cap P_9$, $P_1\cap P_4\cap P_9$ and $P_1\cap P_5\cap P_9$ 
has more than twelve $2$-dimensional faces. 
Thus there are more than three $3$-dimensional faces which have more than twelve $2$-dimensional faces. 
Hence, there is no polyhedron $Q^6$ corresponding to the case $($\ref{sec:3-1-1}$)$. 

\indent We consider the case $($\ref{sec:3-1-2}$)$: 
only one $2$-dimensional face has two cusps. 
We can express this $2$-dimensional face as the intersection of 
four $5$-dimensional faces $P_1$, $P_2$, $P_3$ and $P_4$. 
There are only four $3$-dimensional faces 
$P_1\cap P_2\cap P_3$, $P_1\cap P_2\cap P_4$, $P_1\cap P_3\cap P_4$ and $P_2\cap P_3\cap P_4$ 
which have two cusps. 
Note that each of these $3$-dimensional faces has more than eight $2$-dimensional faces. 
By the same reason as in the case $($\ref{sec:3-1-1}$)$, 
if there are at least twelve $3$-dimensional faces 
which have more than twelve $2$-dimensional faces each, 
then $Q^6$ does not satisfy inequality $($\ref{eq:1}$)$. 
Table 1 is a part of the list of $3$-dimensional faces which have exactly one cusp. 
Any three $3$-dimensional faces in the same row are pairwise adjacent. 
By Lemma \ref{lem:4-2}, at least one of the faces in a row has more than twelve $2$-dimensional faces. 
Thus there exist at least twelve $3$-dimensional faces 
which have more than twelve $2$-dimensional faces each. 
Hence $Q^6$ does not satisfy inequality $($\ref{eq:1}$)$. 

\indent We consider the case $($\ref{sec:3-1-3}$)$: 
exactly one edge connects two cusps. 
Denote this edge by $P_1\cap P_2\cap P_3\cap P_4\cap P_5$. 
Then there are exactly ten $3$-dimensional faces 
$P_1\cap P_2\cap P_3$, $P_1\cap P_2\cap P_4$, $P_1\cap P_2\cap P_5$, 
$P_1\cap P_3\cap P_4$, $P_1\cap P_3\cap P_5$, $P_1\cap P_4\cap P_5$, 
$P_2\cap P_3\cap P_4$, $P_2\cap P_3\cap P_5$, $P_2\cap P_4\cap P_5$ and $P_3\cap P_4\cap P_5$ 
which have two cusps. 
Note that each of these $3$-dimensional faces has at least ten $2$-dimensional faces. 
By the same reason in the cases $($\ref{sec:3-1-1}-\ref{sec:3-1-2}$)$, 
if there are more than twenty $3$-dimensional faces which have more than twelve $2$-dimensional faces each, 
then $Q^6$ does not satisfy inequality $($\ref{eq:1}$)$. 
Table 2 is a part of the list of $3$-dimensional faces 
any three of which in the same row are pairwise adjacent, analogous to Table 1. 

\indent By Lemma \ref{lem:4-2}, there exist at least twenty $3$-dimensional faces 
which have more than twelve $2$-dimensional faces. 
Thus $Q^6$ does not satisfy inequality $($\ref{eq:1}$)$. 

\indent Thus neither of the cases $($\ref{sec:3-1-1}-\ref{sec:3-1-3}$)$ satisfies  inequality $($\ref{eq:1}$)$. 
Thus $c(Q^6)\not=2$. Hence $c(Q^6)\geq 3$. 
\clearpage
\begin{table}[htp]
\begin{center}
\setlength{\tabcolsep}{10pt}
\footnotesize
\begin{tabular}{|l|l|l|}
\hline 
$P_1\cap P_2\cap P_6$ & $P_1\cap P_2\cap P_7$ & $P_1\cap P_2\cap P_8$ \\ \hline 
$P_1\cap P_3\cap P_6$ & $P_1\cap P_3\cap P_7$ & $P_1\cap P_3\cap P_9$ \\ \hline 
$P_1\cap P_4\cap P_6$ & $P_1\cap P_4\cap P_8$ & $P_1\cap P_4\cap P_9$ \\ \hline 
$P_1\cap P_5\cap P_7$ & $P_1\cap P_5\cap P_8$ & $P_1\cap P_5\cap P_9$ \\ \hline 
$P_2\cap P_3\cap P_6$ & $P_2\cap P_3\cap P_7$ & $P_2\cap P_3\cap P_{10}$ \\ \hline 
$P_2\cap P_4\cap P_6$ & $P_2\cap P_4\cap P_8$ & $P_2\cap P_4\cap P_{10}$ \\ \hline 
$P_2\cap P_5\cap P_7$ & $P_2\cap P_5\cap P_8$ & $P_2\cap P_5\cap P_{10}$ \\ \hline 
$P_3\cap P_4\cap P_6$ & $P_3\cap P_4\cap P_9$ & $P_3\cap P_4\cap P_{10}$ \\ \hline 
$P_3\cap P_5\cap P_7$ & $P_3\cap P_5\cap P_9$ & $P_3\cap P_5\cap P_{10}$ \\ \hline 
$P_4\cap P_5\cap P_8$ & $P_4\cap P_5\cap P_9$ & $P_4\cap P_5\cap P_{10}$ \\ \hline 
$P_2\cap P_6\cap P_{10}$ & $P_2\cap P_7\cap P_{10}$ & $P_2\cap P_8\cap P_{10}$ \\ \hline
$P_3\cap P_6\cap P_{10}$ & $P_3\cap P_7\cap P_{10}$ & $P_3\cap P_9\cap P_{10}$ \\ \hline
\end{tabular}
\caption{$3$-dimensional faces which satisfy the conditions of Lemma \ref{lem:4-2} in the case $($\ref{sec:3-1-2}$)$}
\end{center}
\end{table}
\begin{table}[htp]
\begin{center}
\setlength{\tabcolsep}{10pt}
\footnotesize
\begin{tabular}{|l|l|l|}
\hline 
$P_1\cap P_2\cap P_6$ & $P_1\cap P_2\cap P_7$ & $P_1\cap P_2\cap P_8$ \\ \hline 
$P_1\cap P_3\cap P_6$ & $P_1\cap P_3\cap P_7$ & $P_1\cap P_3\cap P_9$ \\ \hline 
$P_1\cap P_4\cap P_6$ & $P_1\cap P_4\cap P_8$ & $P_1\cap P_4\cap P_9$ \\ \hline 
$P_1\cap P_5\cap P_7$ & $P_1\cap P_5\cap P_8$ & $P_1\cap P_5\cap P_9$ \\ \hline 
$P_2\cap P_3\cap P_6$ & $P_2\cap P_3\cap P_7$ & $P_2\cap P_3\cap P_{10}$ \\ \hline 
$P_2\cap P_4\cap P_6$ & $P_2\cap P_4\cap P_8$ & $P_2\cap P_4\cap P_{10}$ \\ \hline 
$P_2\cap P_5\cap P_7$ & $P_2\cap P_5\cap P_8$ & $P_2\cap P_5\cap P_{10}$ \\ \hline 
$P_3\cap P_4\cap P_6$ & $P_3\cap P_4\cap P_9$ & $P_3\cap P_4\cap P_{10}$ \\ \hline 
$P_3\cap P_5\cap P_7$ & $P_3\cap P_5\cap P_9$ & $P_3\cap P_5\cap P_{10}$ \\ \hline 
$P_4\cap P_5\cap P_8$ & $P_4\cap P_5\cap P_9$ & $P_4\cap P_5\cap P_{10}$ \\ \hline 
$P_2\cap P_6\cap P_{10}$ & $P_2\cap P_7\cap P_{10}$ & $P_2\cap P_8\cap P_{10}$ \\ \hline
$P_3\cap P_6\cap P_{10}$ & $P_3\cap P_7\cap P_{10}$ & $P_3\cap P_9\cap P_{10}$ \\ \hline
$P_4\cap P_6\cap P_{10}$ & $P_4\cap P_8\cap P_{10}$ & $P_4\cap P_9\cap P_{10}$ \\ \hline
$P_5\cap P_7\cap P_{10}$ & $P_5\cap P_8\cap P_{10}$ & $P_5\cap P_9\cap P_{10}$ \\ \hline
$P_1\cap P_6\cap P_9$ & $P_1\cap P_7\cap P_9$ & $P_1\cap P_8\cap P_9$ \\ \hline
$P_6\cap P_9\cap P_{10}$ & $P_7\cap P_9\cap P_{10}$ & $P_8\cap P_9\cap P_{10}$ \\ \hline
$P_1\cap P_7\cap P_8$ & $P_2\cap P_7\cap P_8$ & $P_5\cap P_7\cap P_8$ \\ \hline
$P_6\cap P_7\cap P_8$ & $P_7\cap P_8\cap P_9$ & $P_7\cap P_8\cap P_{10}$ \\ \hline
$P_1\cap P_6\cap P_7$ & $P_2\cap P_6\cap P_7$ & $P_3\cap P_6\cap P_7$ \\ \hline
$P_2\cap P_6\cap P_8$ & $P_4\cap P_6\cap P_8$ & $P_6\cap P_8\cap P_{10}$ \\ \hline
\end{tabular}
\caption{$3$-dimensional faces which satisfy the conditions of Lemma \ref{lem:4-2} in the case 
$($\ref{sec:3-1-3}$)$}\end{center}
\end{table}

\section{Proof of the Main Theorem: $n=7$}

\indent Assume that $Q^7$ has exactly $m$ cusps. 
Fix one cusp of $Q^7$, and denote it by $c$. 
It is clear that there are twelve $6$-dimensional faces which share $c$. 
Denote these faces by $E_i$ $(1\leq i\leq 7)$, and assume that 
$E_i$ and $E_j$ are parallel if $i+j=13$. 
Any of the $3$-dimensional faces which have $c$ 
is represented by the intersection of four $6$-dimensional faces which have $c$. 
If a $3$-dimensional face belongs to $E_i$, 
then it is not a $3$-dimensional face of $E_{13-i}$. 
Thus the number of $3$-dimensional faces of $Q^7$ which have $c$ is 
$\frac{12\times 10\times 8\times 6}{4\times 3\times 2\times 1}$, 
i.e. two hundreds and forty $3$-dimensional faces. 
By analogy to Lemma \ref{lem:4-1}, we can prove the following statement. 

\begin{lem}\label{lem:5-1} 
If there are two different $3$-dimensional faces of $Q^7$ which share two cusps, 
then these faces either have a common $2$-dimensional face which has both cusps 
$($but no common edge joins the cusps$)$, 
or have a common edge which starts at one cusp and terminates at the other. 
\end{lem}

Note that any edge of $Q^7$ is given by the intersection of six $6$-dimensional faces of $Q^7$. 
Thus if we fix one edge of $Q^7$, then the number of $3$-dimensional faces of $Q^7$ 
which have this edge is $\binom{6}{4}=15$. 
Thus, by Lemma \ref{lem:5-1}, 
there are no more than fifteen $3$-dimensional faces which have two common cusps. 

\indent Fix one $3$-dimensional face and denote it by $F$. 
Assume that $F$ has $l$ cusps. 
By Lemma \ref{lem:5-1}, there are at most $15\binom{m-l}{2}$ $3$-dimensional faces 
which have more than two cusps which do not belong to $F$. 
By Proposition \ref{pro:2-2}, 
each of these $3$-dimensional faces has at least six $2$-dimensional faces. 
In addition, by Lemma \ref{lem:5-1}, there are at most $15l(m-l)$ $3$-dimensional faces 
which have one common cusp with $F$, and have one cusp which $F$ does not have. 
By the same reason as above, each of these faces has at least six $2$-dimensional faces. 
By Lemma \ref{lem:5-1}, there are at most $14\binom{l}{2}$ $3$-dimensional faces 
which have at least two common cusps with $F$. 

\indent By the second inequality of Proposition \ref{pro:2-2}, 
each of these $3$-dimensional faces has at least six $2$-dimensional faces. 
But if a $3$-dimensional face has exactly six $2$-dimensional faces, 
by the third inequality of Proposition \ref{pro:2-2}, 
it must have more than two cusps. 
Thus if one of those $3$-dimensional faces has only the cusps that $F$ has, 
then this $3$-dimensional face and $F$ have a common $2$-dimensional face 
which has more than two cusps. Denote this $3$-dimensional face by $F'$. 
There are at least three $2$-dimensional faces of $F'$ 
which are parallel to $F\cap F'$. 
In addition, when we look at one cusp of $F\cap F'$, there are two faces 
which are adjacent to $F\cap F'$. 
Because $F'$ has exactly six $2$-dimensional faces, these two $2$-dimensional faces 
which share a cusp with $F\cap F'$ and are adjacent to $F\cap F'$ 
have every cusp of $F\cap F'$. The latter is impossible. 
Thus $F'$ must have more than six $2$-dimensional faces in this case. 
Thus, any $3$-dimensional face which shares all the cusps with 
$F$ has more than six $2$-dimensional faces. 

\indent Fix one cusp of $Q^7$, and denote it by $c_1$. 
In this case, we note that there are at least $(240-15(m-1))$ $3$-dimensional faces 
which have only one cusp $c_1$ because the number of 3-dimensional faces 
which have at least two cusps including $c_1$ is at most $15(m-1)$. 
Because the number of these $3$-dimensional faces 
which have only the cusp $c_1$ is positive, 
$m$ must be smaller than $17$. 
By Corollary \ref{cor:2-1}, 
each of these $3$-dimensional faces has at least twelve $2$-dimensional faces. Thus, 

\begin{eqnarray*} 
a_3^2(Q^7) & \geq & {\scriptstyle \frac{1}{a_3(Q^7)}
\left( 6\times 15\binom{m-l}{2}+6\times 15l(m-l)+7\times 14\binom{l}{2}+12(240-15(m-1))\right. } \\
& & {\scriptstyle \left. + 12(a_3(Q^7)-(15\binom{m-l}{2}+ 15l(m-l) +14\binom{l}{2}+ (240-15(m-1))))\right) } \\ 
& \geq & {\scriptstyle \frac{1}{a_3(Q^7)} \left( (9-3)\times 15\binom{m-l}{2}+(9-3)\times 15l(m-l)+(9-2)\times 
14\binom{l}{2}+ (9+3)(240-15(m-1)) \right. } \\
& & \left. {\scriptstyle +9(a_3(Q^7)-(15\binom{m-l}{2}+ 15l(m-l) +14\binom{l}{2}+ (240-15(m-1))))}\right) \\ 
& = & 9+{\scriptstyle \frac{-3\times 15\binom{m-l}{2}-3\times 15l(m-l)-2\times 14\binom{l}{2}+ 3(240-15(m-1))}{a_3(Q^7)}}. 
\end{eqnarray*}
On the other hand, by Theorem \ref{thm:2-4}, we obtain $a_3^2(Q^7)<9$. 
In order for $Q^7$ to satisfy all the above inequalities, the following must hold: 
\begin{equation*}
-3\times 15\binom{m-l}{2}-3\times 15l(m-l)-2\times 14\binom{l}{2}+3(240-15(m-1))<0. 
\end{equation*}
Fix another cusp of $Q^7$, and denote it by $c_2$. 
The number of $3$-dimensional faces which have exactly one cusp $c_2$ is at least $240-15(m-2)$. 
By analogy to the above, we obtain: 
\begin{equation*}
-3\times 15\binom{m-l}{2}-3\times 15l(m-l)-2\times 14\binom{l}{2}+3(240-15(m-1))+3(240-15(m-2))<0. 
\end{equation*}
By proceeding in this way, because $Q^7$ has $m$ cusps, we see that 
\begin{eqnarray*}
-3\times 15\binom{m-l}{2}-3\times 15l(m-l)-2\times 14\binom{l}{2}  
&+3(240-15(m-1))& \\ &+3(240-15(m-2))& \\ &\cdots & \\ &+3(240-15\times 1)& <0. 
\end{eqnarray*}
We simplify the left-hand side of this inequality, and finally obtain 
\begin{equation*}
17l^2-17l-90m^2+1530m-1440<0. 
\end{equation*}
The left-hand side is an increasing function of $l$. 
Because these exist a face $F$ with $l\geq 2$ cusps, by substituting $l=2$ in the above  inequality, we get 
\begin{equation*}
90m^2-1530m+1406>0. 
\end{equation*}
To satisfy this inequality $m$ must be greater than or equal to 17. A contradiction. 
Hence $c(Q^7)\geq 17$.

\section{Proof of the Main Theorem: $8\leq n\leq 12$}
Our main theorem for $8\leq n\leq 12$ comes from the following lemma. 

\begin{lem}\label{lem:6-1}
For $8\leq n\leq 12$, if any $(n-1)$-dimensional face of $Q^n$ has at least $m$ cusps, then 
\begin{equation*}
c(Q^n)\geq 3m-2n+1. 
\end{equation*}
\end{lem}

\begin{proof}
Let $L$ be an $(n-1)$-dimensional face of $Q^n$. 
The number of $(n-1)$-dimensional faces 
which are parallel to $L$ and share exactly one cusp with $L$ is $c(L)$. 
Denote these $(n-1)$-dimensional faces by $L'$, $L_1$, $L_2$, $\cdots$, $L_{c(L)-1}$. 
Assume that $Q^n$ has $k$ cusps which are not shared by neither $L$ nor $L'$. 
Each $L_i$ $(1\leq i\leq c(L)-1)$ shares at least $c(L_i)-1-k$ cusps with $L'$. 
Because $c(L_i)\geq m$, $k$ is less than or equal to $m-1$. 
On the other hand, the number of $(n-1)$-dimensional faces sharing a cusp of $L'$ is $2(n-1)$. 
Thus, we obtain the following inequality: 
\begin{equation*}
(c(L_1)-1-k)+(c(L_2)-1-k)+\cdots +(c(L_{c(L)-1})-1-k)\leq (2(n-1)-1)(c(L')-1). 
\end{equation*} 
Because each $L_i$ $(1\leq i\leq c(L)-1)$ and $L$ must have at least $m$ cusps, 
we have 
\begin{equation}
(m-1-k)(m-1)\leq (2(n-1)-1)(c(L')-1). 
\end{equation}
By this inequality, we obtain 
\begin{equation*}
k\geq m-1-\frac{(2n-3)(c(L')-1)}{m-1}. 
\end{equation*}
Thus, we get 
\begin{eqnarray*}
c(Q^n) &\geq & c(L)+c(L')-1+k \\
&\geq & m+c(L')-1+m-1-\frac{(2n-3)(c(L')-1)}{m-1} \\
&\geq & 2m-2+\frac{2n-3}{m-1}+\frac{m-2(n-1)}{m-1}c(L'). 
\end{eqnarray*}
Because $Q^7$ has at least seventeen cusps, 
$m$ is greater than or equal to $2(n-1)$ for $8\leq n\leq 12$. 
Since $\frac{m-2(n-1)}{m-1}$ is positive and $L'$ has at least $m$ cusps, we have 
\begin{eqnarray*}
c(Q^n) &\geq & 2m-2+\frac{2n-3}{m-1}+\frac{m-2(n-1)}{m-1}c(L') \\
&\geq & 2m-2+\frac{2n-3}{m-1}+\frac{m-2(n-1)}{m-1}m \\
&= & 3m-2n+1. 
\end{eqnarray*} 
\end{proof}

\begin{flushleft}
\textbf{Acknowledgement} 
The author would like to express his deepest gratitude to Professor Hiroyasu Izeki 
whose comments and suggestions innumerably valuable throughout the course of his study. 
The author would also like to thank Professor Yoshiaki Maeda and Professor Leonid 
Potyagailo who helped him very much throughout the study in this paper. 
The author would like to thank the referee for helpful comments. 
\end{flushleft}


\begin{flushleft}
Jun Nonaka \\
Department of Mathematics, \\
Keio University,\\
3-14-1 Hiyoshi, Kouhoku-ku, Yokohama-shi, Kanagawa, 223-8522, Japan.\\
email: \verb|jun_b_nonaka@yahoo.co.jp|
\end{flushleft}
\end{document}